 \documentclass[final,3p,times]{elsarticle}

\usepackage{graphicx}
\usepackage{subfigure}
\usepackage{bm}
\usepackage{amsmath, amsfonts, amssymb, amsthm, mathtools}
\usepackage{tikz}
\usetikzlibrary{positioning,arrows.meta}
\usepackage{caption}
\usepackage{subcaption}
\usepackage{url} 
\newtheorem{theorem}{Theorem}[section]
\newtheorem{lemma}{Lemma}[section]
\newtheorem{claim}{Claim}[section]
\usepackage[normalem]{ulem}

\begin{document}

\begin{frontmatter}



\title{The only Class~$0$ Flower snark is the smallest\footnote{G.A.B. thanks CAPES
88887.950125/2024-00. A.L.A.M. thanks CAPES
88887.135037/2025-00. F.L.M. thanks 
CNPq 306049/2025-2. C.M.H.F. thanks FAPERJ E-26/204.143/2024, CNPq 404613/2023-3, and CNPq 305356/2021-6. The authors thank INCT-CQA through CNPq 408884/2024-0. The authors used ChatGPT to check grammar, clarify and organize selected passages, and prepare figures in TikZ format. The authors then reviewed and edited all generated content and take full responsibility for the content of the article.}}


\author[pesc]{G. A. Bridi} 
\ead{gabridi@cos.ufrj.br}

\author[pesc]{A. L. A. Martins} 
\ead{andreluis@cos.ufrj.br}

\author[pesc,caxias]{F. L. Marquezino}
\ead{franklin@cos.ufrj.br}

\author[pesc]{C. M. H. Figueiredo}
\ead{celina@cos.ufrj.br}

\address[pesc]{Systems Engineering and Computer Science, Federal University of Rio de Janeiro, Brazil}
\address[caxias]{Duque de Caxias Campus, Federal University of Rio de Janeiro, Brazil}


\begin{abstract}
In graph pebbling, pebbles are distributed across the vertices of a graph. A pebbling move consists of removing two pebbles from one vertex and adding one pebble to an adjacent vertex. The pebbling number is the minimum number $t$ such that, regardless of how the $t$ pebbles are initially placed, it is always possible to perform a sequence of moves that results in at least one pebble placed on any specified target vertex. Graphs whose pebbling number is equal to the number of vertices are called Class~$0$ and provide a challenging set of graphs that resist being characterized. In this note, we answer a question recently proposed by the pioneering study on the pebbling number of snark graphs: we prove that the smallest Flower snark $J_3$ is Class~$0$, establishing that $J_3$ in fact is the only Class~$0$ Flower snark. that systematically reduce the space of pebble configurations. Our bound-based approach shows particular promise for diameter $3$ graphs, where the neighborhood structure remains sufficiently constrained to enable effective analysis.
\end{abstract}




\end{frontmatter}


\section{Introduction}
Graph Pebbling is a combinatorial game that studies the movement and distribution of discrete resources---called pebbles---on graphs. The game has a single type of move, called a pebbling move, in which two pebbles are removed from a vertex to place one on an adjacent vertex. A fundamental parameter in graph pebbling is the pebbling number of a graph, which is the minimum number of pebbles required to guarantee that, from any initial distribution, it is possible to move a pebble to any target vertex. Pebbling numbers are known, for instance, for complete graphs, paths, cycles, trees, and cubes~\cite{survey_pebbling}.

A trivial lower bound for the pebbling number of a graph is its number of vertices, and graphs whose pebbling number is the number of vertices are called Class~$0$. Characterizing Class~$0$ graphs is a central problem in graph pebbling that remains open in general~\cite{Czygrinow, Pachter}. Consequently, considerable effort has been devoted to identifying structural conditions and graph families for which the Class~$0$ property can be established. Clarke et al.~\cite{diameter2} proved that every $3$-connected diameter $2$ graph is Class~$0$. Alcón et al.~\cite{pebbling_split} much later studied the pebbling number of a notable diameter $3$ chordal graph family called split graphs, whose vertex set can be partitioned into a clique and an independent set, and managed to establish that split graphs with minimum degree at least $3$ are Class~$0$. Adauto et al.~\cite{WFL_kneser} recently provided evidence in favor of the conjecture that every Kneser graph is Class~$0$ by proving that the diameter three Kneser graph $K(10, 4)$ is Class~$0$.

In this context, snarks~\cite{survey_snarks} emerge as a particularly interesting graph family to investigate. They form a classical and extensively studied family in graph theory, often arising as challenging examples in graph-theoretical problems. For instance, snarks are potential minimal counterexamples to several important conjectures, including Tutte's 5-flow conjecture, the 1-factor double cover conjecture, and the cycle double cover conjecture. Their pebbling properties, however, remain largely unexplored. The Petersen graph, the smallest snark, is Class~$0$, and to our knowledge, it was the only snark with a known pebbling number. A pioneering study on the pebbling number of snark graphs was performed by Adauto et al.~\cite{WFL_snarks}. Among their results, they establish that the only Class~$0$ snark of girth at least $5$ is the Petersen graph and that a Class~$0$ snark with girth at most $4$ has its number of vertices at most $22$. They also provide lower and upper bounds for the infinite family of Flower snarks, ruling out the Class~0 property for all its members except the smallest one $J_3$, which remained the only unresolved case.

For this remaining case, Adauto et al.~\cite{WFL_snarks} obtained the bounds $12\leq\pi(J_3)\leq13$, leaving a gap of a single pebble between the known upper bound and the Class~$0$ threshold. Their upper bound was obtained using the Weight Function Lemma~\cite{WFL}, a useful technique grounded in integer linear optimization. In their framework, upper bounds are obtained by considering the corresponding dual optimization problem, which is a well-established strategy in the literature (see Refs.~\cite{WFL, pebbling_split, WFL_kneser}). For instance, the aforementioned proof by Adauto et al.~\cite{WFL_kneser} that $K(10, 4)$ is Class~$0$ was carried out using the dual approach. The question of whether $J_3$ is Class~$0$, however, persisted in subsequent work. In particular, Bridi et al.~\cite{lagos2025} refined the dual technique by introducing a heuristic for finding improved dual solutions. Their method improved the previously known upper bounds for Blanuša 1 and Blanuša 2 and for all Flower snarks considered except $J_3$. More importantly, they proved that the upper bound of the smallest Flower snark cannot be further improved using the dual approach, making this technique inherently insufficient to determine whether this graph is Class~$0$. This limitation highlights the intrinsic difficulty of the problem and motivates the need to explore alternative methods. In this sense, the present note settles the remaining case by proving that the smallest Flower snark $J_3$ is indeed Class~$0$ through a novel bound-based approach, in which specific subsets of vertices are used to derive bounds that systematically restrict the possible pebble distributions.

\paragraph{Snarks}
In this note, $G = (V, E)$ refers to a simple connected graph. The number of vertices of $G$ is denoted by $n(G)$. The distance between vertices $u, v \in V(G)$ is referred to as $d(u, v)$. The eccentricity of a vertex $v \in V(G)$ is the maximum distance $d(u, v)$ for any vertex $u \in V(G)$, while the diameter of $G$ is the maximum distance $d(u, v)$ for any pair $u, v \in V(G)$. The girth is the length of the smallest cycle in the graph. The $k$-th neighborhood of a vertex $v$, denoted by $N_k(v)$, is the set of all vertices in a graph that are at an exact distance $k$ from $v$. The first neighborhood $N_1(v)$ is often simply called the neighborhood of $v$ and denoted by $N(r)$.

Graphs that are cubic, bridgeless, and $4$-edge-chromatic are known as snarks~\cite{survey_snarks}. The snark family plays a significant role in the context of the Four Color Theorem, which is equivalent to the statement that no snark is planar~\cite{fourColors}. For a detailed history of the snark graphs, see Ref.~\cite{survey_snarks}. The Flower snarks $J_m$~\cite{Campos} are defined for odd integers $m = 2k + 1$ with $m \geq 3$. The number of vertices of $J_m$ is $4m$, and its diameter is $k + 2$. In particular, the smallest Flower $J_3$, shown in Figure~\ref{fig:J3}, has $12$ vertices and a diameter of $3$. Note that $J_3$ is a simple modification of the Petersen graph in which one of its vertices is replaced by a triangle. This graph is also referred to as Tietze’s graph~\cite{Tietze}.

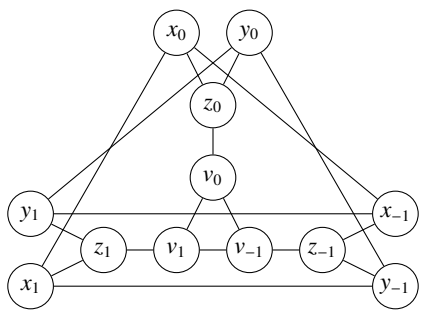
\begin{figure}[htbp]
    \centering
\begin{tikzpicture}[scale=0.48]
    \tikzstyle{every node}=[circle, draw, minimum size=0.6cm, inner sep=0pt]

  \node (A) at (0, 2) {\small $z_0$};
  \node (B) at (0, 0) {\small $v_0$};
  \node (C) at (-1, 4) {\small $x_0$};
  \node (D) at (1, 4) {\small $y_0$};
  \node (E) at (-1, -2) {\small $v_1$};
  \node (F) at (1, -2) {\small $v_{-1}$};
  \node (G) at (-5, -3) {\small $x_1$};
  \node (H) at (5, -1) {\small $x_{-1}$};
  \node (I) at (-5, -1) {\small $y_1$};
  \node (J) at (5, -3) {\small $y_{-1}$};
  \node (K) at (-3, -2) {\small $z_1$};
  \node (L) at (3, -2) {\small $z_{-1}$};
  
  \draw (A) -- (B);
  \draw (A) -- (C);
  \draw (A) -- (D);
  \draw (B) -- (E);
  \draw (B) -- (F);
  \draw (C) -- (G);
  \draw (C) -- (H);
  \draw (D) -- (I);
  \draw (D) -- (J);
  \draw (E) -- (K);
  \draw (F) -- (L);
  \draw (E) -- (F);
  \draw (G) -- (K);
  \draw (I) -- (K);
  \draw (H) -- (L);
  \draw (J) -- (L);
  \draw (H) -- (I);
  \draw (G) -- (J);

\end{tikzpicture}
\caption{The smallest Flower snark $J_3$.}
    \label{fig:J3}
\end{figure}

\paragraph{Graph pebbling}
A configuration $C$ on a graph $G$ is a function that assigns to each vertex $v \in V(G)$ a number $C(v) \in \mathbb{N}$ representing the number of pebbles at that vertex. We define $C(X)$ for a set $X \subseteq V(G)$ as the sum of pebbles on the vertices of $X$, i.e., $C(X) = \sum_{v \in X} C(v)$. A pebbling move removes two pebbles from a vertex $v$ and places one pebble on a neighbor $u \in N(v)$. If there is a combination of pebbling moves that place a pebble on a target vertex $r \in V(G)$, we say that the configuration $C$ is $r$-solvable, while otherwise, $C$ is said to be $r$-unsolvable. The pebbling number $\pi(G)$ of $G$ is the minimum number of pebbles $t$ such that for each target $r \in V(G)$ all configurations with $t$ pebbles are $r$-solvable. The pebbling number is lower bounded by $\pi(G) \geq n(G)$, and we say that $G$ is Class~$0$ if $\pi(G) = n(G)$.

\section{A new Class~$0$ graph}

In order to show that $J_3$ is Class~$0$, it is enough to show that for each possible target $r$, all configurations with $n(J_3) = 12$ pebbles are $r$-solvable. It is sufficient to consider one representative target from each vertex orbit under the automorphisms of $J_3$, and since $J_3$ has three vertex orbits, we consider, without loss of generality, the targets $z_0$, $x_0$, and $v_0$~\cite{WFL_snarks}.

Observe, however, that the number of possible configurations obtained by distributing the pebbles among the vertices of the graph can be very large even for a relatively small graph like $J_3$. In the case of the Class~$0$ proof for $J_3$---assuming that $C(r)=0$, as required for any $r$-unsolvable configuration---a standard stars-and-bars argument gives $646{,}646$ possible configurations
for each fixed target. A first reduction of this large configuration space follows directly from the distances to the target. Indeed, for any vertex $v$, an $r$-unsolvable configuration must satisfy $C(v)\leq 2^{d(v,r)}-1$. 
In our case, Figure~\ref{fig:neighborhoods}
shows representations of $J_3$ for the three representative targets $z_0$, $x_0$, and $v_0$, with the vertices grouped according to their distance from the target. For $r=z_0$ and $r=x_0$, the distance layers contain $3$, $6$, and $2$ vertices, while for $r=v_0$ they contain $3$, $4$, and $4$, respectively. Applying the corresponding bounds of $1$, $3$, and $7$ pebbles per vertex and using the inclusion-exclusion principle reduces the number of possible configurations to $90{,}088$ for $z_0$ and $x_0$, and $126{,}584$ for $v_0$.



These numbers are still prohibitively large for a direct analysis of individual configurations. To circumvent this difficulty, our strategy is to exploit the structure of $J_3$ relative to each target to drastically reduce this number by aggregating configurations that can be treated under the same analysis. Specifically, rather than working with individual vertices, we partition the vertex set into suitably chosen subsets. We then derive bounds on the pebble distributions over these subsets and combine them with controlled pebbling transfers between the subsets to obtain stronger restrictions. Together with the fixed total of $12$ pebbles, the resulting system of inequalities reduces the tens of thousands of individual configurations to only a handful of aggregate cases. As a result, only a few residual cases require more specific approaches, which we handle through suitable pebbling moves adapted to their particular configurations.

\begin{theorem} \label{thm}
The Flower graph $J_3$ is Class~$0$, i.e., the pebbling number of $J_3$ is given by $\pi(J_3) = n(J_3) = 12$.
\end{theorem}

\begin{proof}

\begin{figure}[!t]
    \centering
    \hspace{-0.075\textwidth}
    \begin{minipage}{0.4\textwidth}
        \centering
        \subfigure[]{
            \begin{tikzpicture}[scale=0.45]

            \tikzstyle{main node}=[circle, draw, minimum size=0.6cm, inner sep=0pt]

          \node[main node] (A) at (5, 0) {\small $z_0$};
          \node[main node] (B) at (1, -2.5) {\small $v_0$};
          \node[main node] (C) at (5, -2.5) {\small $x_0$};
          \node[main node] (D) at (9, -2.5) {\small $y_0$};
          \node[main node] (E) at (0, -5.75) {\small $v_1$};
          \node[main node] (F) at (2, -5.75) {\small $v_{-1}$};
          \node[main node] (G) at (4, -5.75) {\small $x_1$};
          \node[main node] (H) at (6, -5.75) {\small $x_{-1}$};
          \node[main node] (I) at (8, -5.75) {\small $y_1$};
          \node[main node] (J) at (10, -5.75) {\small $y_{-1}$};
          \node[main node] (K) at (0, -9) {\small $z_1$};
          \node[main node] (L) at (2, -9) {\small $z_{-1}$};
          
          \draw (A) -- (B);
          \draw (A) -- (C);
          \draw (A) -- (D);
          \draw (B) -- (E);
          \draw (B) -- (F);
          \draw (C) -- (G);
          \draw (C) -- (H);
          \draw (D) -- (I);
          \draw (D) -- (J);
          \draw (E) -- (K);
          \draw (F) -- (L);
          \draw (E) -- (F);
          \draw (G) -- (K);
          \draw (I) -- (K);
          \draw (H) -- (L);
          \draw (J) -- (L);
          \draw (H) -- (I);
          \draw (G) -- (4, -4.5);
          \draw (J) -- (10, -4.5);
          \draw (4, -4.5) -- (10, -4.5);
          
          \draw[dashed] (-1, -1.25) -- (13.75, -1.25);
          \draw[dashed] (-1, -3.75) -- (13.75, -3.75);
          \draw[dashed] (-1, -7.75) -- (13.75, -7.75);
        
          \node at (12.5, 0) {\small $z_0$};
          \node at (12.5, -2.5) {\small $N_1(z_0)$};
          \node at (12.5, -5.75) {\small $N_2(z_0)$};
          \node at (12.5, -9) {\small $N_3(z_0)$};
            
            \end{tikzpicture}
        }
    \end{minipage}
    \hspace{0.06\textwidth}
    \begin{minipage}{0.4\textwidth}
        \centering
        \subfigure[]{
            \begin{tikzpicture}[scale=0.45]

            \tikzstyle{main node}=[circle, draw, minimum size=0.6cm, inner sep=0pt]

          \node[main node] (A) at (5, 0) {\small $x_0$};
          \node[main node] (B) at (1, -2.5) {\small $z_0$};
          \node[main node] (C) at (5, -2.5) {\small $x_1$};
          \node[main node] (D) at (9, -2.5) {\small $x_{-1}$};
          \node[main node] (E) at (0, -5.75) {\small $v_0$};
          \node[main node] (F) at (2, -5.75) {\small $y_0$};
          \node[main node] (G) at (4, -5.75) {\small $z_{1}$};
          \node[main node] (H) at (6, -5.75) {\small $y_{-1}$};
          \node[main node] (I) at (8, -5.75) {\small $z_{-1}$};
          \node[main node] (J) at (10, -5.75) {\small $y_1$};
          \node[main node] (K) at (-1, -9) {\small $v_1$};
          \node[main node] (L) at (1, -9) {\small $v_{-1}$};
          
          \draw (A) -- (B);
          \draw (A) -- (C);
          \draw (A) -- (D);
          \draw (B) -- (E);
          \draw (B) -- (F);
          \draw (C) -- (G);
          \draw (C) -- (H);
          \draw (D) -- (I);
          \draw (D) -- (J);
          \draw (E) -- (K);
          \draw (E) -- (L);
          \draw (K) -- (L);
          \draw (H) -- (I);
          \draw (G) -- (K);
          \draw (I) -- (L);
          \draw (F) -- (2, -4.75);
          \draw (H) -- (6, -4.75);
          \draw (2, -4.75) -- (6, -4.75);
          \draw (F) -- (2, -7);
          \draw (J) -- (10, -7);
          \draw (2, -7) -- (10, -7);
          \draw (G) -- (4, -4.25);
          \draw (J) -- (10, -4.25);
          \draw (4, -4.25) -- (10, -4.25);
        
          \draw[dashed] (-2, -1.25) -- (13.75, -1.25);
          \draw[dashed] (-2, -3.75) -- (13.75, -3.75);
          \draw[dashed] (-2, -7.75) -- (13.75, -7.75);
        
          \node at (12.5, 0) {\small $x_0$};
          \node at (12.5, -2.5) {\small $N_1(x_0)$};
          \node at (12.5, -5.5) {\small $N_2(x_0)$};
          \node at (12.5, -9) {\small $N_3(x_0)$};
            
            \end{tikzpicture}
        }
    \end{minipage}

    \vspace{0.5cm} 
    \begin{minipage}{0.6\textwidth}
        \centering
        \subfigure[]{
            \begin{tikzpicture}[scale=0.45]

            \tikzstyle{main node}=[circle, draw, minimum size=0.6cm, inner sep=0pt]

              \node[main node] (A) at (5, 0) {\small $v_0$};
              \node[main node] (B) at (1, -2.5) {\small $z_0$};
              \node[main node] (C) at (5, -2.5) {\small $v_1$};
              \node[main node] (D) at (9, -2.5) {\small $v_{-1}$};
              \node[main node] (E) at (0, -5) {\small $x_0$};
              \node[main node] (F) at (2, -5) {\small $y_0$};
              \node[main node] (G) at (5, -5) {\small $z_{1}$};
              \node[main node] (H) at (9, -5) {\small $z_{-1}$};
              \node[main node] (I) at (4, -8.25) {\small $x_1$};
              \node[main node] (J) at (6, -8.25) {\small $y_1$};
              \node[main node] (K) at (8, -8.25) {\small $x_{-1}$};
              \node[main node] (L) at (10, -8.25) {\small $y_{-1}$};
              
              \draw (A) -- (B);
              \draw (A) -- (C);
              \draw (A) -- (D);
              \draw (B) -- (E);
              \draw (B) -- (F);
              \draw (C) -- (G);
              \draw (D) -- (H);
              \draw (G) -- (I);
              \draw (G) -- (J);
              \draw (H) -- (K);
              \draw (H) -- (L);
              \draw (C) -- (D);
              \draw (E) -- (I);
              \draw (E) -- (K);
              \draw (F) -- (J);
              \draw (F) -- (L);
              \draw (J) -- (K);
              \draw (I) -- (4, -9.25);
              \draw (L) -- (10, -9.25);
              \draw (4, -9.25) -- (10, -9.25);
              
              \draw[dashed] (-1, -1.25) -- (13.75, -1.25);
              \draw[dashed] (-1, -3.75) -- (13.75, -3.75);
              \draw[dashed] (-1, -6.25) -- (13.75, -6.25);
            
              \node at (12.5, 0) {\small $v_0$};
              \node at (12.5, -2.5) {\small $N_1(v_0)$};
              \node at (12.5, -5) {\small $N_2(v_0)$};
              \node at (12.5, -8.25) {\small $N_3(v_0)$};
            
            \end{tikzpicture}
        }
    \end{minipage}

    \caption{Neighborhood representations of the graph $J_3$ 
    according to the choice of target vertex $r$: (a)~$z_0$, (b)~$x_0$, and (c)~$v_0$. In each subfigure, vertices are placed according to their distance with respect to $r$, with dashed lines indicating the neighborhood levels $N_1(r)$, $N_2(r)$, and $N_3(r)$.}
    \label{fig:neighborhoods}
\end{figure}
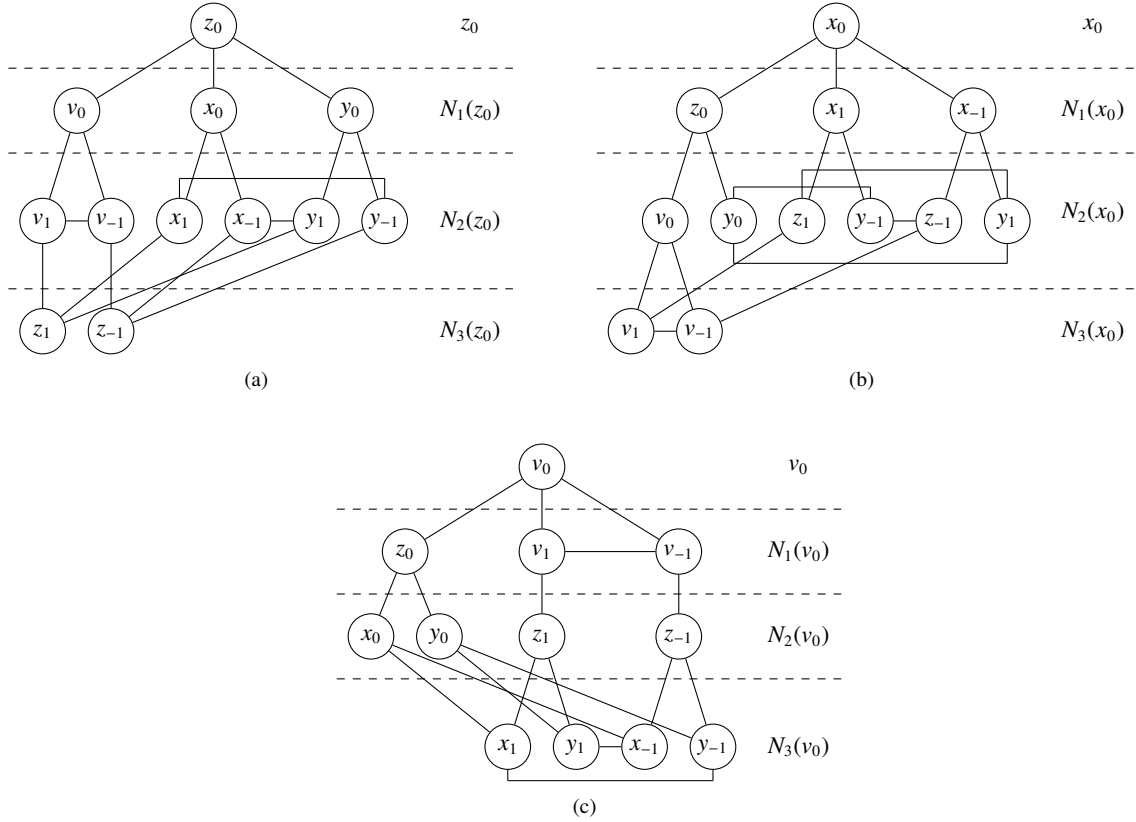

For each target $r=z_0,x_0,v_0$, we consider all possible $r$-unsolvable configurations with $12$ pebbles and show that none of them can occur. We first consider the targets $r=z_0$ and $r=x_0$ jointly. Please refer to Figures~\ref{fig:neighborhoods}(a) and~\ref{fig:neighborhoods}(b). These two targets share a common structure that allows us to carry out a joint analysis up to a certain point. We first exploit this common structure and then consider the remaining cases separately, using structural properties specific to each target. For each of these two targets, we partition the vertices different from $r$ into three subsets $D$, $E$, and $F$. The corresponding partition for each target is given in Table~\ref{tab:z_0x_0_partitions}.

\begin{table}[ht]
\centering
\caption{Subsets $D$, $E$, and $F$ for the targets $z_0$ and $x_0$.}
\label{tab:z_0x_0_partitions}
\begin{tabular}{c|c|c|c}
\hline
$r$ & $D$ & $E$ & $F$ \\
\hline
$z_0$
& $\{x_0,x_1,x_{-1},y_0,y_1,y_{-1}\}$
& $\{v_0,v_1,v_{-1}\}$
& $\{z_1,z_{-1}\}$ \\[2mm]
$x_0$
& $\{x_1,z_1,y_{-1},x_{-1},z_{-1},y_1\}$
& $\{z_0,v_0,y_0\}$
& $\{v_1,v_{-1}\}$ \\
\hline
\end{tabular}
\end{table}

The common structure of these two partitions lies in the connections of the sets $D$ and $E$ with $F$. Observe that an $r$-unsolvable configuration must satisfy $C(E) \leq 4$. The internal structure of $D$ is such that its vertices in $N_2(r)$ are pairwise connected by a matching of size two. This connection implies that an $r$-unsolvable configuration satisfies $C(D)\leq 6$. Moreover, each vertex of $F$ is adjacent to a vertex in $N_2(r)\cap D$ and to a vertex in $N_2(r)\cap E$. This connectivity yields the following transfer property.
\begin{claim}\label{claim:z0x0_F}
If $C(F)\geq 2k+1$, where $k\in\mathbb{N}$, then $k$ pebbles can be transferred from $F$ to either $D$ or $E$.
\end{claim}
By combining Claim~\ref{claim:z0x0_F} with $C(E) \leq 4$ and $C(D) \leq 6$, we get the bounds 
\begin{equation}
    \label{eq:EF}
    C(F) + 2 \ C(E) \leq 10, 
\end{equation}
and
\begin{equation}
    \label{eq:DF}
    C(F) + 2 \ C(D) \leq 14.
\end{equation}
Now, we consider these bounds together with the total number of pebbles. Precisely, we obtain the following system
\begin{equation}
\label{eq:system}
\begin{cases}
C(D)+C(E)+C(F)=12,\\
C(F)+2C(E)\leq 10,\\
C(F)+2C(D)\leq 14,\\
C(D)\leq 6.
\end{cases}
\end{equation}
The possible aggregate distributions of pebbles among the sets $D$, $E$, and $F$ satisfying System~\eqref{eq:system} are reduced to five, as proved by the following lemma.
\begin{lemma} \label{lemma:z0x0system}
For the targets $z_0$ and $x_0$, the only possible aggregated distributions $(C(D),C(E),C(F))$ satisfying  System~\eqref{eq:system} are
\[
(6,4,2),\quad
(5,3,4),\quad
(4,2,6),\quad
(3,1,8),\quad
(2,0,10).
\]
\end{lemma}

\begin{proof}
By adding the second and third inequalities of System~\eqref{eq:system} and replacing the first equation, we can conclude that both inequalities must hold with equality. Hence, $C(E)=C(D)-2$ and $C(F)=14-2C(D)$. Since $C(D)\leq6$ and $C(E)\geq0$, we have $2\leq C(D)\leq6$, which yields exactly the five distributions stated above.
\end{proof}

As established in the proof of Lemma~\ref{lemma:z0x0system}, both bounds in Eqs.~\eqref{eq:EF} and~\eqref{eq:DF} hold with equality for each of the five distributions. Thus, any additional pebble that can be transferred to the corresponding set is sufficient to solve the configuration. In order to exploit this property, we prove first that we can further constrain an unsolvable configuration. 

\begin{lemma}\label{lemma:z0x0N1} For $r\in\{z_0,x_0\}$, any $r$-unsolvable configuration satisfies $C(N_1(r)\cap E)=0$. 
\end{lemma} 
\begin{proof} Suppose that $C(N_1(r)\cap E)=1$. In this case, $C(E)\leq 3$, and Eq.~\eqref{eq:EF} can be strengthened to $C(F)+2C(E)\leq 8$. Together with Eq.~\eqref{eq:DF}, this implies $2C(D)+2C(E)+2C(F)\leq 22$. However, since $C(D)+C(E)+C(F)=12$, we have $2C(D)+2C(E)+2C(F)=24$, a contradiction. Therefore, $C(N_1(r)\cap E)=0$, proving the lemma.
\end{proof}
So, let $C(N_1(r)\cap E)=0$. Let's consider the connections between $E$ and $D\cup F$. In particular, every vertex of $N_2(r)\cap E$ is adjacent to a vertex of $D\cup F$. This yields the following transfer property.
\begin{claim}\label{claim:z0x0_E}
If $C(E)\geq 2k+1$, where $k\in\mathbb{N}$, then $k$ pebbles can be transferred from $E$ to $D\cup F$.
\end{claim}
Thus, Claim~\ref{claim:z0x0_E} solves the distributions $(6,4,2)$ and $(5,3,4)$, since in both cases one pebble can be transferred from $E$ to $D\cup F$, violating Eq.~\eqref{eq:DF}. That way, only three $(C(D),C(E),C(F))$ cases remain: $(4,2,6)$, $(3,1,8)$, and $(2,0,10)$. From this point on, the analysis must be particularized for each target, since the arguments used to analyze the remaining distributions rely on structural properties specific to each target. 

We first consider the target $z_0$. Please recall Figure~\ref{fig:neighborhoods}(a). For this target, we refine the set $D$ as $D=A\cup B$, where $A=\{x_0,x_1,x_{-1}\}$ and $B=\{y_0,y_1,y_{-1}\}$. Similarly to the bound $C(E)\leq4$, we must have $C(A)\leq4$ and $C(B)\leq4$ for $r$-unsolvable configurations. Moreover, the relation previously established between $F$ and $D$ also holds between $F$ and each of the sets $A$ and $B$ individually. Specifically, each vertex of $F$ is adjacent to a vertex of $A$ and to a vertex of $B$. Thus, Claim~\ref{claim:z0x0_F} also holds with $D$ replaced by either $A$ or $B$. That way, analogously to Eq.~\eqref{eq:EF}, we obtain
\begin{equation}
    \label{eq:XF}
    C(F) + 2 \ C(X) \leq 10, \ \ X \in \{A, B \}.
\end{equation}
Adding the inequalities obtained with $X=A$ and $X=B$, we obtain $C(F)+C(D)\leq 10$. The aggregate distributions $(3,1,8)$ and $(2,0,10)$ do not satisfy this inequality. Therefore, only the distribution $(4,2,6)$ remains for the target $z_0$. This remaining case lies on the boundary of the inequalities established above, so a more specialized argument is required to complete the analysis.
\begin{itemize}
    \item $(4,2,6)$ Since Eqs.~\eqref{eq:EF} and~\eqref{eq:DF} both hold with equality, we must have $C(x)\leq1$ for every $x\in N_2(r)$; otherwise, a pebble could be transferred to $F$, violating one of these bounds. By Lemma~\ref{lemma:z0x0N1}, $C(N_1(r)\cap E)=0$. Similarly, we can exclude pebbles in $N_1(r)\cap A$ and $N_1(r)\cap B$. Indeed, if either set contained a pebble in $N_1(r)$, its corresponding bound $C(F)+2C(X)\leq10$ would strengthen to $8$. Adding this strengthened bound to the bound for the other set yields $C(F)+C(D)\leq9$, contradicting $C(F)+C(D)=10$ for the distribution $(4,2,6)$. Combining these observations, we can further conclude that $C(x)=1$ for every $x\in N_2(r)$. Finally, possibly with the help of the pebbles on the set $B$, we send from $F$ one pebble to each of the vertices of $N_2(r) \cap A$, solving this case.
\end{itemize}

\begin{figure}[ht]
\centering
\scalebox{0.925}{

\begin{tikzpicture}[
    node distance=12mm and 22mm,
    box/.style={
        draw,
        rounded corners,
        align=center,
        inner sep=6pt,
        font=\small
    },
    arrow/.style={
        -{Stealth},
        thick
    },
    label/.style={
        align=center,
        font=\scriptsize
    }
]

\node[box] (large)
{Large configuration space\\
$90{,}088$ configurations per target};

\node[box, below=14mm of large] (five)
{$(6,4,2),\ (5,3,4),\ (4,2,6),\ (3,1,8),\ (2,0,10)$};

\node[box, below=14mm of five] (three)
{$(4,2,6),\ (3,1,8),\ (2,0,10)$};

\node[box, below left=18mm and 18mm of three] (z0)
{$(4,2,6)$};

\node[box, below right=18mm and 18mm of three] (x0)
{$(4,2,6),\ (3,1,8),\ (2,0,10)$};

\node[box, below=44mm of three] (solved)
{All cases solved};

\draw[arrow] (large) -- node[right, label]
{Structural partition and System~\eqref{eq:system}\\} (five);

\draw[arrow] (five) -- node[right, label]
{Common transfer arguments\\
eliminate $(6,4,2)$ and $(5,3,4)$} (three);

\draw[arrow] (three) -- node[left, label, xshift=-5mm]
{$r=z_0$\\
$C(F)+C(D)\leq10$\\
eliminates $(3,1,8)$ and $(2,0,10)$} (z0);

\draw[arrow] (three) -- node[right, label, xshift=2mm]
{$r=x_0$} (x0);

\draw[arrow] (z0) |- node[left, label, pos=0.25]
{Target-specific\\
argument} (solved);

\draw[arrow] (x0) |- node[right, label, pos=0.25]
{Target-specific\\
arguments} (solved);

\end{tikzpicture}
}

\caption{Reduction of the configuration space for the targets
$z_0$ and $x_0$.}
\label{fig:configuration_reduction}

\end{figure}

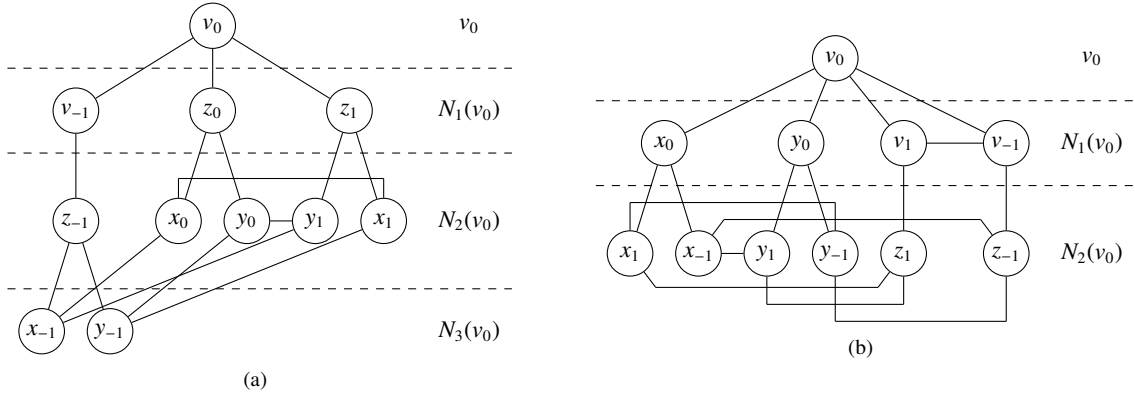
\begin{figure}[!t]
    \centering
    \hspace{-0.075\textwidth}
    \begin{minipage}{0.4\textwidth}
        \centering
        \subfigure[]{
            \begin{tikzpicture}[scale=0.45]

          \tikzstyle{main node}=[circle, draw, minimum size=0.6cm, inner sep=0pt]

          \node[main node] (A) at (5, 0) {\small $v_0$};
          \node[main node] (B) at (5, -2.5) {\small $z_0$};
          \node[main node] (D) at (1, -2.5) {\small $v_{-1}$};
          \node[main node] (E) at (4, -5.75) {\small $x_0$};
          \node[main node] (F) at (6, -5.75) {\small $y_0$};
          \node[main node] (G) at (9, -2.5) {\small $z_{1}$};
          \node[main node] (H) at (1, -5.75) {\small $z_{-1}$};
          \node[main node] (I) at (8, -5.75) {\small $y_1$};
          \node[main node] (J) at (10, -5.75) {\small $x_1$};
          \node[main node] (K) at (0, -9) {\small $x_{-1}$};
          \node[main node] (L) at (2, -9) {\small $y_{-1}$};
          
          \draw (A) -- (B);
          \draw (A) -- (G);
          \draw (A) -- (D);
                  \draw (B) -- (E);
          \draw (B) -- (F);
          \draw (D) -- (H);
          \draw (G) -- (J);
          \draw (G) -- (I);
          \draw (H) -- (K);
          \draw (H) -- (L);
          \draw (E) -- (4, -4.5);
          \draw (J) -- (10, -4.5);
          \draw (4, -4.5) -- (10, -4.5);
          \draw (E) -- (K);
          \draw (F) -- (I);
          \draw (F) -- (L);
          \draw (I) -- (K);
          \draw (J) -- (L);
        
          \draw[dashed] (-1, -1.25) -- (13.75, -1.25);
          \draw[dashed] (-1, -3.75) -- (13.75, -3.75);
          \draw[dashed] (-1, -7.75) -- (13.75, -7.75);
        
          \node at (12.5, 0) {\small $v_0$};
          \node at (12.5, -2.5) {\small $N_1(v_0)$};
          \node at (12.5, -5.75) {\small $N_2(v_0)$};
          \node at (12.5, -9) {\small $N_3(v_0)$};
            
            \end{tikzpicture}
        }
    \end{minipage}
    \hspace{0.06\textwidth}
    \begin{minipage}{0.4\textwidth}
        \centering
        \subfigure[]{
            \begin{tikzpicture}[scale=0.45]

          \tikzstyle{main node}=[circle, draw, minimum size=0.6cm, inner sep=0pt]

          \node[main node] (A) at (5, 0) {\small $v_0$};
          \node[main node] (C) at (7, -2.5) {\small $v_1$};
          \node[main node] (D) at (10, -2.5) {\small $v_{-1}$};
          \node[main node] (E) at (0, -2.5) {\small $x_0$};
          \node[main node] (F) at (4, -2.5) {\small $y_0$};
          \node[main node] (G) at (7, -5.75) {\small $z_{1}$};
          \node[main node] (H) at (10, -5.75) {\small $z_{-1}$};
          \node[main node] (I) at (-1, -5.75) {\small $x_1$};
          \node[main node] (J) at (3, -5.75) {\small $y_1$};
          \node[main node] (K) at (1, -5.75) {\small $x_{-1}$};
          \node[main node] (L) at (5, -5.75) {\small $y_{-1}$};
          
          \draw (A) -- (C);
          \draw (A) -- (D);
          \draw (A) -- (E);
          \draw (A) -- (F);
          \draw (C) -- (D);
          \draw (C) -- (G);
          \draw (D) -- (H);
          \draw (E) -- (I);
          \draw (E) -- (K);
          \draw (F) -- (J);
          \draw (F) -- (L);
          \draw (J) -- (K);
          \draw (I) -- (-1, -4.25);
          \draw (L) -- (5, -4.25);
          \draw (-1, -4.25) -- (5, -4.25);
          \draw (H) -- (10, -7.75);
          \draw (L) -- (5, -7.75);
          \draw (5, -7.75) -- (10, -7.75);
          \draw (G) -- (7, -7.25);
          \draw (J) -- (3, -7.25);
          \draw (3, -7.25) -- (7, -7.25);
          \draw (H) -- (9.25, -4.75);
          \draw (K) -- (1.75, -4.75);
          \draw (1.75, -4.75) -- (9.25, -4.75);
          \draw (G) -- (6.25, -6.75);
          \draw (I) -- (-0.25, -6.75);
          \draw (-0.25, -6.75) -- (6.25, -6.75);
            
          \draw[dashed] (-2, -1.25) -- (13.75, -1.25);
          \draw[dashed] (-2, -3.75) -- (13.75, -3.75);
        
          \node at (12.5, 0) {\small $v_0$};
          \node at (12.5, -2.5) {\small $N_1(v_0)$};
          \node at (12.5, -5.75) {\small $N_2(v_0)$};
            
            \end{tikzpicture}
        }
    \end{minipage}
    \caption{Neighborhood representation of the reduced version graphs used to solve the cases (a) $C(v_1) = 1$ and (b) $C(z_0) = 1$ 
    for the target $v_0$.}
    \label{fig:reduced_v0}
\end{figure}

Now, we consider the target $x_0$. Please recall Figure~\ref{fig:neighborhoods}(b). For this target, the sets $A$ and $B$ are given by $A=\{x_1,z_1,y_{-1}\}$ and $B=\{x_{-1},z_{-1},y_1\}$. In contrast to the case $r=z_0$, where each vertex of $F$ is adjacent to a vertex in both $N_2(r)\cap A$ and $N_2(r)\cap B$, here each vertex of $F$ is adjacent to a vertex in only one of these sets: one is connected to $N_2(r)\cap A$, while the other is connected to $N_2(r)\cap B$. This restricts the possible transfers from $F$ to $A$ and $B$, so Eq.~\eqref{eq:XF} does not hold in this case. Consequently, the distributions $(2,0,10)$ and $(3,1,8)$ cannot be eliminated in the same way as for the target $z_0$. Thus, the three remaining distributions require more specialized pebbling moves, as listed below.
\begin{itemize}
    \item $(4,2,6)$ The proof follows the same general structure as for $z_0$, but some steps require different arguments. As before, we have $C(x)\leq1$ for every $x\in N_2(r)$, and, by Lemma~\ref{lemma:z0x0N1}, $C(N_1(r)\cap E)=0$. To prove that $C(x)=0$ for every $x\in N_1(r)$, by symmetry, it is enough to consider $x\in A$. Suppose that $C(x_1)=1$. If $C(z_1)=1$, we can send a second pebble to $z_1$ from $F$. Otherwise, $C(y_{-1})=1$. In this case, if $C(v_1)\geq2$, possibly with the help of the edge $v_1v_{-1}$, we can gather at least $4$ pebbles at $v_1$ and send $2$ pebbles to $A$. If $C(v_1)\leq1$, we can send a second pebble to $y_{-1}$ with the help of $z_{-1}$. Thus, $C(x)=0$ for every $x\in N_1(r)$. As in the case $r=z_0$, it follows that $C(x)=1$ for every $x\in N_2(r)$. Finally, in addition to the possible help of the pebbles on $B$, we may need the help of the edge $v_1v_{-1}$ to send through $F$ one pebble to each of the vertices of $N_2(r)\cap A$, solving this case.

    \item $(3,1,8)$  The same argument used in the configuration $(4,2,6)$ applies here and, up to symmetry between $A$ and $B$, yields $C(x)=0$ for every $x\in N_1(r)$ and $C(z_1)=C(y_{-1})=1$. Now, we send $2$ pebbles to $A \cup B$ from $F$. With the remaining $4$ pebbles in $F$, we can send a pebble to which set among $A$ and $B$ is convenient to gather---possibly with the help of the edge $y_{-1} z_{-1}$---either $4$ pebbles on $z_{1}$ or $2$ pebbles on both $z_{1}$ and $y_{-1}$. 

    \item ($2,0,10$) Since $C(F) = 10$, we can move $4$ pebbles to the vertex $v_0 \in N_2(r)$, which solves this case.
\end{itemize}
The reduction of the configuration space throughout the proof is summarized in Figure~\ref{fig:configuration_reduction}.

Finally, the last target to be analyzed is $r = v_0$. Please refer to Figure~\ref{fig:neighborhoods}(c). We apply the same strategy of deriving inequalities and using them to reduce and aggregate the possible configurations. However, unlike the previous two targets, target $v_0$ has four vertices in the third neighborhood instead of only two, which increases the number of remaining cases and the complexity of the specific arguments required to solve them. In this way, from the outset, we begin by restricting our analysis, as we did by using Lemma~\ref{lemma:z0x0N1}, and we eliminate the possibility of pebbles in the first neighborhood of $v_0$, as established by the following lemma.

\begin{lemma}\label{lemma:v0N1}
Any $v_0$-unsolvable configuration satisfies $C(N_1(v_0))=0$.
\end{lemma}

\begin{proof}
We begin assuming $C(v_1) = 1$. Firstly, observe that if we add one more pebble to $v_1$, then the configuration is $r$-solvable. Therefore, the neighbors of $v_1$ (except $v_0$) act as neighbors of $v_0$ from the point of view of the pebbling problem. That way, we can eliminate the vertex $v_1$ and link the vertex $z_1$ to $v_0$, resulting in the graph of Figure~\ref{fig:reduced_v0}(a). Let $A = \{ z_{0}, x_{0}, y_0 \}$, $B = \{ z_{1}, x_{1}, y_1 \}$, $E = \{ v_{-1}, z_{-1} \}$, and $F = \{ x_{-1}, y_{-1} \}$. The relationship between all the defined sets is precisely the same as in the case of the target $z_0$. The only difference between the current case and the target $z_0$ case is that now the set $E$ has only a single vertex on $N_2(r)$ instead of two. However, as the number of pebbles decreased to $11$, $C(E) \leq 3$ plays the role of the bound $C(E) \leq 4$ on the analysis of target $z_0$. In this way, all the arguments used on target $z_0$ are directly transferable to this reduced graph, which solves the case $C(v_1) = 1$. The case $C(v_{-1}) = 1$ is solved by symmetry. 

Now, if $C(z_0) = 1$, by analogous arguments of $C(v_1) = 1$ case, we can reduce $J_3$ to the graph given by Figure~\ref{fig:reduced_v0}(b). We define $A = \{ x_0, x_{1}, x_{-1} \}$, $B = \{ y_0, y_{1}, y_{-1} \}$, $E = \{ v_1, z_{1} \}$, and $F = \{ v_{-1}, z_{-1} \}$. Observe that $C(A \cup B), C(E \cup F) \leq 6$. On the other hand, since we need to place $11$ pebbles on the reduced graph, it follows that $C(A \cup B), C(E \cup F) \geq 5$. In this way, either $C(E \cup F) = 5$ or $C(E \cup F) = 6$. For each case, we can send, respectively, $1$ or $2$ pebbles from $E \cup F$ to $A \cup B$, contradicting $C(A \cup B)  \leq 6$, which solves the case $C(z_0) = 1$.
\end{proof}

By Lemma~\ref{lemma:v0N1}, we may set $C(z_0)=C(v_1)=C(v_{-1})=0$. Now, we return to the original graph with $12$ pebbles. Please refer again to Figure~\ref{fig:neighborhoods}(c). Now, let $A = \{z_1, x_{1}, y_{1} \}$, $B = \{z_{-1}, x_{-1}, y_{-1} \}$, $E = \{ x_{0}, y_{0} \}$. The structure here is quite different from the previous targets $z_0$ and $x_0$. In particular, sets $A$ and $B$ are farther away from the target, in such a way that $C(X) \leq 8$, where $X \in \{A, B \}$. On the other hand, $C(E) \leq 4$ still holds here. Now, we exploit the connectivity between the sets to establish a transfer property. Particularly, every vertex of $A \cap N_3(r)$ and $B \cap N_3(r)$ is adjacent to a vertex in each of the other sets. This allows us to establish the following transfer property.
\begin{claim}\label{claim:v0}
If $C(X)\geq2k+2$, where $X\in\{A,B\}$ and $k\in\mathbb{N}$, then $k$ pebbles can be transferred through $X$ to either of the other sets in $\{A,B,E\}\setminus X$.
\end{claim}
In contrast to Claims~\ref{claim:z0x0_E} and~\ref{claim:z0x0_F}, here we cannot exclude the possibility of a pebble on the vertex of $X$ closest to the target. Therefore, the transfer condition requires $2k+2$ pebbles instead of $2k+1$. Observe also that this property is applied in the regime considered here, where $C(z_{-1}),C(z_1)\leq3$, since $z_{-1},z_1\in N_2(r)$. As a direct consequence of Claim~\ref{claim:v0}, together with the bound $C(E)\leq4$, any $r$-unsolvable configuration must satisfy
\begin{equation}
    \label{eq:XE}
    C(X) + 2 \ C(E) \leq 11, \ \ X \in \{A, B \}.
\end{equation}
Equation~\eqref{eq:XE} provides the basis for the system of constraints considered in this case, which is precisely
\begin{equation}
\label{eq:v0system}
\begin{cases}
C(A)+C(B)+C(E)=12,\\
C(A)+2C(E)\leq11,\\
C(B)+2C(E)\leq11,\\
C(A)\leq8,\\
C(A)\geq C(B),
\end{cases}
\end{equation}
where the last inequality follows from the symmetry between $A$ and $B$, without loss of generality. The possible aggregate distributions satisfying this system are characterized in the following lemma.
\begin{lemma}\label{lemma:v0system}
For the target $v_0$, the only possible aggregate distributions $(C(A),C(B),C(E))$ satisfying System~\eqref{eq:v0system} are
\[
(5,4,3),\quad (5,5,2),\quad (6,4,2),\quad (7,3,2),\quad (6,5,1),\quad (7,4,1),\quad (8,3,1),\quad (6,6,0),\quad (7,5,0),\quad (8,4,0).
\]
\end{lemma}

\begin{proof}
Adding the second and third inequalities of System~\eqref{eq:v0system} and using the first equation gives $C(E)\leq3$. We then substitute $C(E)=0,1,2,3$ and, taking into account that $C(A)\geq C(B)$ and $C(A)\leq8$, verify which configurations satisfy the second inequality, yielding the distributions stated in the lemma.
\end{proof}

As an additional step, we identified a common reduction that allows several of the remaining distributions to be treated together. Through suitable pebbling moves, these distributions can reach a common reduced configuration, with some already having more pebbles than required. We formalize this reduction in the following lemma.

\begin{lemma}\label{lemma:v0reduced}
Any reduced distribution satisfying $C(A)=4$, $C(B)=3$, and $C(E)=3$ is $v_0$-solvable.
\end{lemma}
\begin{proof}
If we are able to get two pebbles at some vertex of $N_3(r)\cap B$, we can transfer one pebble to $E$, contradicting Eq.~\eqref{eq:XE} applied with $X=A$. The only situations in which this is not possible are when $C(b)=1$ for every $b\in B$ or $C(z_{-1})=3$. Now, we show that both remaining situations are $r$-solvable using the following observations: firstly, we can get a pebble at any vertex of $N_3(r) \cap B$ that we choose; secondly, since $C(E) = 3$, there exists a vertex $e \in E$ such that if it receives one more pebble, then we have a $r$-solvable configuration; thirdly, since $C(A) = 4$, we can get $2$ pebbles at some vertex $a \in N_3(r) \cap A$. That way, we can send the desired pebble to $e$ through the edge $ae$ if $a$ and $e$ are adjacent or, otherwise, use a vertex $N_3(r) \cap B$, whichever is convenient, to transfer the pebble to $e$.
\end{proof}
Lemma~\ref{lemma:v0reduced} applies to seven of the distributions identified above. Using the transfer property stated in Claim~\ref{claim:v0}, these distributions can be reduced as follows.
\begin{itemize}
    \item $(5,4,3)$ already has more pebbles than required by the hypotheses of the lemma;
    \item $(6,4,2)$ reduces to $(4,4,3)$ by sending one pebble from $A$ to $E$;
    \item $(5,5,2)$ and $(7,3,2)$ reduce to $(5,3,3)$ by sending one pebble to $E$ respectively from $B$ and $A$;
    \item $(6,5,1)$ reduces to $(4,3,3)$ by sending one pebble to $E$ from each of $A$ and $B$;
    \item $(7,4,1)$ and $(8,3,1)$ reduce, up to symmetry between $A$ and $B$, to $(4,3,3)$ by sending two pebbles from $A$ to~$E$. 
\end{itemize}
The remaining three distributions are analyzed separately. Before proceeding, we make two observations. First, for $X\in\{A,B\}$, $C(X)\leq7$ whenever $C(x)\geq1$ for some $x\in N_2(r)\cap X$. Second, if $C(x)=0$ for the same vertex $x$, the bound of $2k+2$ in Claim~\ref{claim:v0} reduces to $2k+1$. We now proceed as follows.
\begin{itemize}
    \item ($8,4,0$) We can send a ninth pebble to $A$ from $B$, contradicting $C(A) \leq 8$.
    
    \item ($7,5,0$) Since we can send an eighth pebble to $A$ from $B$, we must have $C(z_{1}) = 0$. Thus, we can send $3$ pebbles to $B$ from $A$ and similarly conclude that $C(z_{-1}) = 0$. Therefore, we can send $5$ pebbles to $E$, being $3$ from $A$ and $2$ from $B$.
    
    \item ($6,6,0$) Since we can send $2$ pebbles from $B$ to $A$, we can conclude that $C(z_{1}) = 0$. Symmetrically, $C(z_{-1}) = 0$. Now, we send $3$ pebbles to $E$, being $1$ from $A$ and $2$ from $B$. Observe that now we have $C(A) = 4$, $C(B) = 2$, and $C(E) = 3$. If $C(b) = 2$ for some $b \in N_3(r) \cap B$, we can send a fourth pebble to $E$ from $B$ and thus contradict Eq.~\eqref{eq:XE} applied with set $X = A$. Therefore, $C(b) = 1 \ \forall b \in N_3(r) \cap B$. Now, as we argued in Lemma~\ref{lemma:v0reduced}, we may define auxiliary vertices $a$ and $e$, and transfer a pebble from $a$ to $e$ to solve this case.
\end{itemize}

The analysis of $z_0$, $x_0$, and $v_0$ proves that every $12$-pebble configuration is $r$-solvable for all $r\in V(J_3)$.
\end{proof}

\begin{figure}[t]
    \centering
    \hspace{-0.075\textwidth}
    \begin{minipage}{0.5\textwidth}
        \centering
        \subfigure[]{
            \begin{tikzpicture}[scale=0.95]

          \tikzstyle{main node}=[circle, draw, minimum size=0.6cm, inner sep=0pt]

            \node (E1)[main node] at (-6.00, 2.00) {\small $r$};
            \node (E2)[main node] at (-7.90, 0.62) {\small $e_0$};
            \node (E3)[main node] at (-7.18, -1.62) {\small $e_1$};
            \node (E4)[main node] at (-4.82, -1.62) {\small $b_{-1}$};
            \node (E5)[main node] at (-4.10, 0.62) {\small $b_0$};
            \node (F1)[main node] at (-6.00, 1.00) {\small $a_0$};
            \node (F2)[main node] at (-6.95, 0.31) {\small $e_{-1}$};
            \node (F3)[main node] at (-6.59, -0.81) {\small $a_{-1}$};
            \node (F4)[main node] at (-5.41, -0.81) {\small $a_1$};
            \node (F5)[main node] at (-5.05, 0.31) {\small $b_1$};
    
            \draw (E1) -- (E2);
            \draw (E2) -- (E3);
            \draw (E3) -- (E4);
            \draw (E4) -- (E5);
            \draw (E5) -- (E1);
            \draw (F1) -- (F3);
            \draw (F3) -- (F5);
            \draw (F5) -- (F2);
            \draw (F2) -- (F4);
            \draw (F4) -- (F1);
            \draw (E1) -- (F1);
            \draw (E2) -- (F2);
            \draw (E3) -- (F3);
            \draw (E4) -- (F4);
            \draw (E5) -- (F5);
            
            \end{tikzpicture}
        }
    \end{minipage}
    \begin{minipage}{0.4\textwidth}
        \centering
        \subfigure[]{
            \begin{tikzpicture}[scale=0.45]

        \tikzstyle{main node}=[circle, draw, minimum size=0.6cm, inner sep=0pt]

          \node[main node] (A) at (5, 0) {\small $r$};
          \node[main node] (B) at (1, -2.5) {\small $e_0$};
          \node[main node] (C) at (5, -2.5) {\small $a_0$};
          \node[main node] (D) at (9, -2.5) {\small $b_0$};
          \node[main node] (E) at (0, -6.25) {\small $e_1$};
          \node[main node] (F) at (2, -6.25) {\small $e_{-1}$};
          \node[main node] (G) at (4, -6.25) {\small $a_1$};
          \node[main node] (H) at (6, -6.25) {\small $a_{-1}$};
          \node[main node] (I) at (8, -6.25) {\small $b_1$};
          \node[main node] (J) at (10, -6.25) {\small $b_{-1}$};
          
          \draw (A) -- (B);
          \draw (A) -- (C);
          \draw (A) -- (D);
          \draw (B) -- (E);
          \draw (B) -- (F);
          \draw (C) -- (G);
          \draw (C) -- (H);
          \draw (D) -- (I);
          \draw (D) -- (J);
          \draw (F) -- (G);
          \draw (H) -- (I);
          \draw (E) -- (0, -7.5);
          \draw (H) -- (6, -7.5);
          \draw (0, -7.5) -- (6, -7.5);
          \draw (G) -- (4, -8);
          \draw (J) -- (10, -8);
          \draw (4, -8) -- (10, -8);
          \draw (F) -- (2, -5);
          \draw (I) -- (8, -5);
          \draw (2, -5) -- (8, -5);
          \draw (E) -- (0, -4.5);
          \draw (J) -- (10, -4.5);
          \draw (0, -4.5) -- (10, -4.5);
          
          \draw[dashed] (-1, -1.25) -- (13.75, -1.25);
          \draw[dashed] (-1, -3.75) -- (13.75, -3.75);
        
          \node at (12.5, 0) {\small $r$};
          \node at (12.5, -2.5) {\small $N_1(r)$};
          \node at (12.5, -6.25) {\small $N_2(r)$};
            
            \end{tikzpicture}
        }
    \end{minipage}
    \caption{Proof that Petersen graph $P$ is Class~$0$ by our bound-based approach. Subfigure~\ref{fig:petersen}(a) shows the standard representation of the Petersen graph, while Subfigure~\ref{fig:petersen}(b) shows its neighborhood representation. Petersen is a vertex-transitive graph, reducing the analysis to one target $r$. We assume by contradiction that every configuration with $n(P) = 10$ pebbles is $r$-unsolvable. Let $A = \{a_0, a_{-1}, a_{1} \}$, $B = \{b_0, b_{-1}, b_{1} \}$, $E = \{e_0, e_{-1}, e_{1} \}$, and $X, Y, Z \in \{A, B, E \}, Y \neq Z$. We have $C(X) \leq 4$. For any pair $Y, Z$, we have the same connection between $Y$ and $Z$ that we had between $A$ and $B$ for the targets $z_0$ and $x_0$ of $J_3$, i.e., every vertex of $N_2(r) \cap Y$ is adjacent to one vertex of $N_2(r) \cap Z$, and vice versa. So, it follows that $C(Y \cup Z) \leq 6$. We assume without loss of generality that $C(A) \geq C(B), C(E)$. Since we need to place $10$ pebbles on the vertices of the graph, we must have $C(A) = 4$. On the other hand, since $C(Y \cup Z) \leq 6$, we must have $C(B), C(E) \leq 2$. Therefore, we can never place $10$ pebbles in an $r$-unsolvable configuration.}
    \label{fig:petersen}
\end{figure}
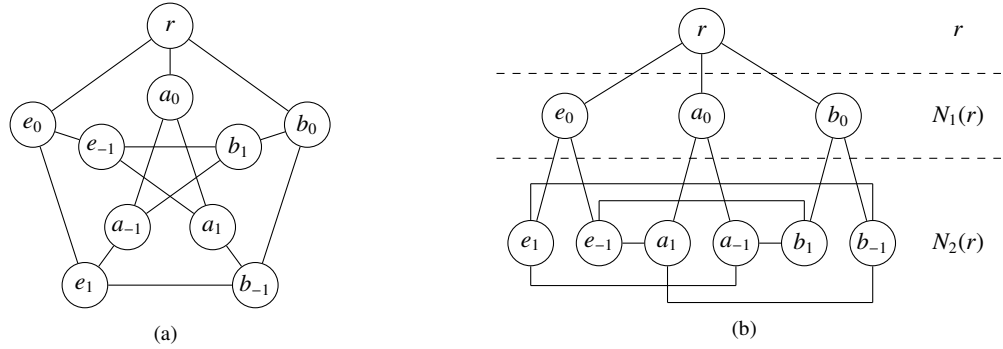

\begin{figure}[!t]
    \centering
    \hspace{-0.01\textwidth}
    \begin{minipage}{0.3\textwidth}
        \centering
        \subfigure[]{
            \begin{tikzpicture}[scale=0.35]

            \tikzstyle{main node}=[circle, draw, minimum size=0.4cm, inner sep=0pt]

              \node[main node, label=above:{\scriptsize $r$}] (A) at (5, 0) {\scriptsize $1$};
              \node[main node, label=above:{\scriptsize $0$}] (B) at (1, -2.5) {\scriptsize $2$};
              \node[main node, label=left:{\scriptsize $0$}] (C) at (5, -2.5) {\scriptsize $3$};
              \node[main node, label=above:{\scriptsize $0$}] (D) at (10.75, -2.5) {\scriptsize $4$};
              \node[main node, label=left:{\scriptsize $1$}] (E) at (0, -5) {\scriptsize $10$};
              \node[main node, label=right:{\scriptsize $0$}] (F) at (2, -5) {\scriptsize $9$};
              \node[main node, label=right:{\scriptsize $0$}] (G) at (5, -5) {\scriptsize $12$};
              \node[main node, label=left:{\scriptsize $1$}] (H) at (10.75, -5) {\scriptsize $11$};
              \node[main node, label=below:{\scriptsize $0$}] (I) at (4, -8.25) {\scriptsize $8$};
              \node[main node, label=right:{\scriptsize $7$}] (J) at (6, -8.25) {\scriptsize $6$};
              \node[main node, label=left:{\scriptsize $3$}] (K) at (9.75, -8.25) {\scriptsize $7$};
              \node[main node, label=right:{\scriptsize $0$}] (L) at (11.75, -8.25) {\scriptsize $5$};
              
              \draw (A) -- (B);
              \draw (A) -- (C);
              \draw (A) -- (D);
              \draw (B) -- (E);
              \draw (B) -- (F);
              \draw (C) -- (G);
              \draw (D) -- (H);
              \draw (G) -- (I);
              \draw (G) -- (J);
              \draw (H) -- (K);
              \draw (H) -- (L);
              \draw (C) -- (D);
              \draw (E) -- (I);
              \draw (F) -- (K);
              \draw (E) -- (F);
              \draw (I) -- (J);
              \draw (K) -- (L);
              \draw (J) -- (6, -9.25);
              \draw (L) -- (11.75, -9.25);
              \draw (6, -9.25) -- (11.75, -9.25);
            
            \end{tikzpicture}
        }
    \end{minipage}
    \hspace{0.03\textwidth}
    \begin{minipage}{0.3\textwidth}
        \centering
        \subfigure[]{
            \begin{tikzpicture}[scale=0.35]

          \tikzstyle{main node}=[circle, draw, minimum size=0.4cm, inner sep=0pt]

              \node[main node, label=above:{\scriptsize $r$}] (A) at (5, 0) {\scriptsize $6$};
              \node[main node, label=above:{\scriptsize $0$}] (B) at (1, -2.5) {\scriptsize $8$};
              \node[main node, label=left:{\scriptsize $0$}] (C) at (5, -2.5) {\scriptsize $5$};
              \node[main node, label=above:{\scriptsize $0$}] (D) at (9, -2.5) {\scriptsize $9$};
              \node[main node, label=left:{\scriptsize $0$}] (E) at (0, -5) {\scriptsize $12$};
              \node[main node, label=below:{\scriptsize $0$}] (F) at (2, -5) {\scriptsize $10$};
              \node[main node, label=right:{\scriptsize $0$}] (G) at (5, -5) {\scriptsize $7$};
              \node[main node, label=left:{\scriptsize $0$}] (H) at (9, -5) {\scriptsize $4$};
              \node[main node, label=left:{\scriptsize $1$}] (I) at (-1.5, -8.25) {\scriptsize $2$};
              \node[main node, label=left:{\scriptsize $1$}] (J) at (1.5, -8.25) {\scriptsize $3$};
              \node[main node, label=right:{\scriptsize $5$}] (K) at (5, -8.25) {\scriptsize $11$};
              \node[main node, label=left:{\scriptsize $5$}] (L) at (9, -8.25) {\scriptsize $1$};
              
              \draw (A) -- (B);
              \draw (A) -- (C);
              \draw (A) -- (D);
              \draw (B) -- (E);
              \draw (B) -- (F);
              \draw (C) -- (G);
              \draw (D) -- (H);
              \draw (C) -- (D);
              \draw (G) -- (K);
              \draw (H) -- (L);
              \draw (F) -- (G);
              \draw (E) -- (I);
              \draw (E) -- (J);
              \draw (J) -- (K);
              \draw (F) -- (3.25, -3.75);
              \draw (3.25, -3.75) -- (7.75, -3.75);
              \draw (7.75, -3.75) -- (H);
              \draw (I) -- (-0.25, -9.25);
              \draw (-0.25, -9.25) -- (3.75, -9.25);
              \draw (3.75, -9.25) -- (K);
              \draw (J) -- (3, -7.25);
              \draw (3, -7.25) -- (7.5, -7.25);
              \draw (7.5, -7.25) -- (L);
              \draw (I) -- (-1.5, -9.75);
              \draw (-1.5, -9.75) -- (9, -9.75);
              \draw (9, -9.75) -- (L);
            
            \end{tikzpicture}
        }
    \end{minipage}
    \hspace{0.02\textwidth}
    \begin{minipage}{0.3\textwidth}
        \centering
        \subfigure[]{
            \begin{tikzpicture}[scale=0.35]

            \tikzstyle{main node}=[circle, draw, minimum size=0.35cm, inner sep=0pt]

              \node[main node, label=above:{\scriptsize $r$}] (A) at (5, 0) {\scriptsize $1$};
              \node[main node, label=above:{\scriptsize $0$}] (B) at (1, -2.5) {\scriptsize $4$};
              \node[main node, label=left:{\scriptsize $0$}] (C) at (5, -2.5) {\scriptsize $3$};
              \node[main node, label=above:{\scriptsize $0$}] (D) at (11, -2.5) {\scriptsize $2$};
              \node[main node, label=left:{\scriptsize $0$}] (E) at (0, -5) {\scriptsize $9$};
              \node[main node, label=right:{\scriptsize $0$}] (F) at (2, -5) {\scriptsize $8$};
              \node[main node, label=right:{\scriptsize $1$}] (G) at (5, -5) {\scriptsize $11$};
              \node[main node, label=above:{\scriptsize $1$}] (H) at (8, -5) {\scriptsize $12$};
              \node[main node, label=left:{\scriptsize $1$}] (I) at (11, -5) {\scriptsize $10$};
              \node[main node, label=left:{\scriptsize $7$}] (J) at (0, -8.25) {\scriptsize $5$};
              \node[main node, label=right:{\scriptsize $1$}] (K) at (5, -8.25) {\scriptsize $7$};
              \node[main node, label=left:{\scriptsize $1$}] (L) at (11, -8.25) {\scriptsize $6$};
              
              \draw (A) -- (B);
              \draw (A) -- (C);
              \draw (A) -- (D);
              \draw (B) -- (E);
              \draw (B) -- (F);
              \draw (C) -- (G);
              \draw (C) -- (H);
              \draw (D) -- (H);
              \draw (E) -- (F);
              \draw (D) -- (I);
              \draw (E) -- (J);
              \draw (G) -- (K);
              \draw (I) -- (L);
              \draw (J) -- (K);
              \draw (G) -- (L);
              \draw (I) -- (K);
              \draw (F) -- (2, -6.25);
              \draw (2, -6.25) -- (8, -6.25);
              \draw (8, -6.25) -- (H);
              \draw (J) -- (0, -9.25);
              \draw (0, -9.25) -- (11, -9.25);
              \draw (11, -9.25) -- (L);
            
            \end{tikzpicture}
        }
    \end{minipage}

    \caption{Unsolvable configurations with $12$ pebbles for three cubic graphs with $12$ vertices, girth~$3$, and diameter~$3$: (a) graph \#1395 (truncated tetrahedral graph), (b) graph \#44170, and (c) graph \#44172. The target vertex of each graph is indicated by $r$. The numbers inside the vertices correspond to their labels in the House of Graphs database, while the values outside indicate the number of pebbles assigned to each vertex. Unsolvability was verified through an exhaustive search over all possible pebbling move combinations, using a tree as the underlying data structure. The software used in this search is available at https://github.com/gabridi/pebbling\_unsolvability.}
    \label{fig:unsolvables}
\end{figure}
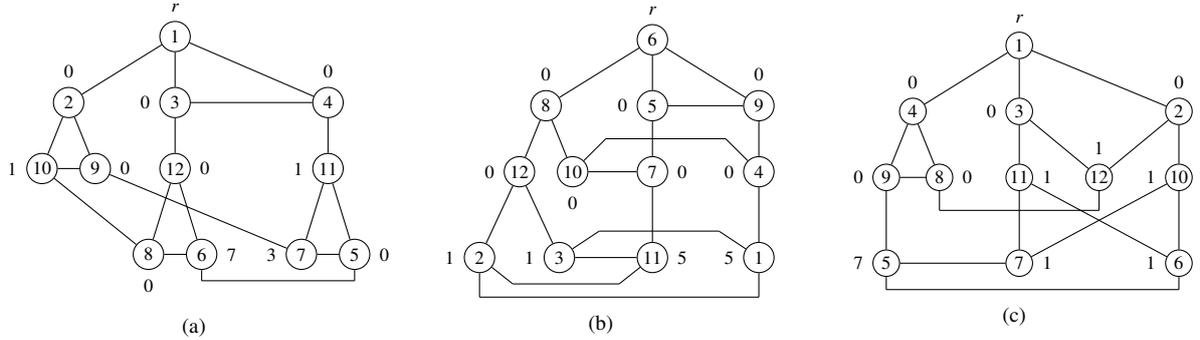

\paragraph{Final remarks}
The effectiveness of our bound-based approach seems to be closely related to the diameter of the graph---more precisely, to the eccentricity of each target vertex. Graphs with diameter $2$, such as the Petersen graph (see Figure~\ref{fig:petersen}) and the reduced graph of Figure~\ref{fig:reduced_v0}(b) allow for a concise and direct analysis. On the other hand, the demonstration for $J_3$, a diameter $3$ graph, was more challenging, demanding a much more detailed analysis. Beyond diameter, the study of $J_3$ suggests that the number of vertices in the outermost neighborhood also plays a role. In particular, the target vertex $v_0$, with four vertices at distance $3$, required greater effort than $z_0$ and $x_0$---which have only two---including a separate analysis of the cases where each of its neighbors has exactly one pebble. This evidences that the effectiveness of our method depends not only on the eccentricity of the target but also on the structural depth around the target, i.e., how many vertices lie in each neighborhood. These difficulties are intuitive, as more distant vertices make it harder to establish bounds and efficiently move pebbles toward the target. Extending our method beyond diameter~$3$ is unclear, as deeper structures challenge the current bounding strategy. This remains open for future investigation. It is worth noting that, although deciding on the solvability of a given configuration is NP-complete~\cite{pebbling_completeness1, pebbling_completeness2}, this does not preclude the applicability of our method to infinite families of well-structured graphs----for instance, by defining bounds and move strategies in a recursive manner.

While extending our method beyond diameter $3$ graphs remains uncertain, our findings indicate great potential for diameter $3$ graphs. A characterization for the pebbling number of a graph with diameter~$2$ is already established in the literature (see Ref.~\cite{diameter2}). In contrast, what we know about diameter $3$ graphs is that their pebbling number is upper bounded by $3n(G)/2 + \mathcal{O}(1)$~\cite{diameter3}. In this way, within the diameter $3$ context, much remains to be explored, and our approach may provide a valuable tool for advancing the understanding of the pebbling number of these graphs. As a theoretical direction, the structural bounds developed here suggest the possibility of identifying sufficient conditions for graphs of diameter~$3$ to be Class~$0$. In particular, it would be interesting to investigate which properties of the neighborhood structure are sufficient for the bound-based reduction used in this work to establish the Class~$0$ property. As a concrete direction, one natural direction is to consider other cubic graphs with $12$ vertices, girth~$3$, and diameter~$3$—--the same parameters as the Flower snark $J_3$. According to the House of Graphs database~\cite{house_of_graphs}, only five such graphs exist. Besides $J_3$, the others include graphs with IDs\footnote{These identifiers refer to entries in the \emph{House of Graphs} online database: \url{https://houseofgraphs.org}.} \#1395, \#6698, \#44170, and \#44172. The graph \#1395 is known as the truncated tetrahedral graph. We find unsolvable configurations with $12$ pebbles for the graphs \#1395, \#44170, and \#44172 (see Figure~\ref{fig:unsolvables}), implying that these graphs are not Class~$0$. One could investigate whether the pebbling number of these three graphs is indeed $13$; and whether graph \#6698 is Class~$0$, i.e., determine whether $J_3$ is the only Class~$0$ graph among the five with the same parameters. 

Finally, it is important to mention that although our method was used to prove that $J_3$ is Class~$0$, what our method actually establishes is that every configuration with $x$ pebbles is $r$-solvable for every target vertex $r$ of an arbitrary graph $G$—thereby yielding the upper bound for the pebbling number. The larger the gap between $x$ and the actual pebbling number, the easier the proof tends to be. For instance, in the case of $J_3$, one could verify that the targets $z_0$ and $x_0$ can be fully solved using only Eq.~\eqref{eq:EF} applied with set $X = E$ and Eq.~\eqref{eq:DF} if $x \geq 13$. Thus, our method can be particularly useful for tightening the upper bounds of graphs with a large window for the pebbling number.


\bibliographystyle{elsarticle-num}
\bibliography{references}

\end{document}